\documentclass[a4paper]{article}

\usepackage{amsmath}
\usepackage{amssymb}
\usepackage{amsthm}
\usepackage{upgreek}
\usepackage[hyphens]{url}
\usepackage{graphicx}
\usepackage{hyperref}
\usepackage{breakurl}
\usepackage{float}
\usepackage{verbatim}
\usepackage{pbox}
\usepackage{listings}

\newtheorem{theorem}{Theorem}
\newtheorem{lemma}{Lemma}
\newtheorem{corollary}{Corollary}
\newtheorem{remark}{Remark}

\newcommand{\raisedot}{\raisebox{2pt}{$.$}}
\newcommand{\raisecomma}{\raisebox{2pt}{$,$}}

\begin{document}

\title{A bound for the error term in the Brent-McMillan algorithm}

\author{
  Richard P. Brent
  \footnote{Mathematical Sciences Institute, 
  Australian National University, Canberra, 
  Australia [\texttt{gamma@rpbrent.com}];
  supported by Australian Research Council grant DP140101417.}\\
  \and
  Fredrik Johansson
  \footnote{
  RISC, Johannes Kepler University, 4040 Linz, Austria
  [\texttt{fredrik.johansson@risc.jku.at}]; supported by the Austrian Science
  Fund (FWF) grant Y464-N18.}
}
\date{}

\maketitle

\begin{abstract}
The Brent-McMillan algorithm B3 (1980), when implemented with binary splitting,
is the fastest known algorithm for high-precision computation of Euler's
constant. However, no rigorous error bound for the algorithm has ever been
published.  We provide such a bound and justify the empirical observations
of Brent and McMillan.  We also give bounds on the error
in the asymptotic expansions of functions related to the Bessel functions
$I_0(x)$ and $K_0(x)$ for positive real~$x$.
\end{abstract}

\section{Introduction}

Brent and McMillan \cite{BrentMcMillan1980,Jackson1996} observed
that Euler's constant
\begin{equation*}
\gamma = \lim_{n \rightarrow \infty } (H_n - \ln(n)) \approx 0.5772156649,
	\quad H_n = \sum_{k=1}^n \frac{1}{k}\,\raisecomma
\end{equation*}
can be computed rapidly to high accuracy using the formula
\begin{equation}
\gamma = \frac{S_0(2n) - K_0(2n)}{I_0(2n)} - \ln(n)\,\raisecomma
\label{eq:bessel}
\end{equation}
where $n > 0$ is a free parameter (understood to be an integer),
$K_0(x)$ and $I_0(x)$ denote the usual Bessel functions, and
\begin{equation*}
S_0(x) = \sum_{k=0}^{\infty} \frac{H_k}{(k!)^2} \left(\frac{x}{2}\right)^{2k}.
\end{equation*}
The idea is to choose $n$ optimally so that an asymptotic
series can be used to compute $K_0(2n)$, while $S_0(2n)$ and $I_0(2n)$
are computed using Taylor series.

When all series are evaluated using the \emph{binary splitting}
technique (see \cite[\S4.9]{BrentZimmermann2010}), 
the first $d$ digits of $\gamma$
can be computed in essentially optimal time $O(d^{1+\varepsilon})$.
This approach has been used for all recent record
calculations of $\gamma$, including the
current world record of 29,844,489,545 digits
set by A. Yee and R. Chan in 2009 \cite{Yee}.

Brent and McMillan gave three algorithms (B1, B2 and B3)
to compute $\gamma$ via~\eqref{eq:bessel}.
The most efficient, B3, approximates
$K_0(2n)$ using the asymptotic expansion
\begin{equation}
2x I_0(x) K_0(x) = \sum_{k=0}^{m/2-1} \frac{b_k}{x^{2k}} + T_m(x)\,, \quad
	b_k = \frac{[(2k)!]^3}{(k!)^4 8^{2k}}\,\raisecomma
\label{eq:ikasymp}
\end{equation}
where one should take $m \approx 4n$.
The expansion \eqref{eq:ikasymp}
appears as formula 9.7.5 in Abramowitz and Stegun \cite{AbramowitzStegun1964},
and 10.40.6 in the Digital Library of Mathematical Functions \cite{DLMF}.
Unfortunately, neither work gives a proof or reference, and no bound
for the error term $T_m(x)$ is provided. Brent and McMillan
observed empirically that $T_{4n}(2n) = O(e^{-4n})$, which
would give a final error of $O(e^{-8n})$ for~$\gamma$,
but left this as a conjecture.

Brent \cite{Brent2010} recently noted that the error
term can be bounded rigorously,
starting from the individual asymptotic expansions of $I_0(x)$ and $K_0(x)$.
However, he did not present an explicit bound at that time.
In this paper, we calculate an explicit error bound,
allowing the fastest version of the Brent-McMillan algorithm (B3) to be used
for provably correct evaluation of $\gamma$.

To bound the error in the Brent-McMillan algorithm we must bound the
errors in evaluating the transcendental functions $I_0(2n)$, $K_0(2n)$ and
$S_0(2n)$ occurring in~\eqref{eq:bessel} (we ignore the
error in evaluating $\ln(n)$ since this is well-understood). The most
difficult task is to bound the error associated with $K_0(2n)$. For
reasons of efficiency, the algorithm approximates
$I_0(2n)K_0(2n)$ using the asymptotic expansion~\eqref{eq:ikasymp},
and then the term $K_0(2n)/I_0(2n)$ in~\eqref{eq:bessel} is computed from
$I_0(2n)K_0(2n)/I_0(2n)^2$.

Sections~\ref{sec:bounds}--\ref{sec:product} contain
bounds on the size of various error terms
that are needed for the main result.
For example, Lemma~\ref{lemma:Q_mbound} bounds the error in the asymptotic
expansion for $I_0(x)$, which is nontrivial as the terms do not have
alternating signs.

The asymptotic expansion~\eqref{eq:ikasymp} can be obtained formally by
multiplying the asymptotic
expansions (see~\eqref{eq:k0series}--\eqref{eq:i0series} below) for
$K_0$ and $I_0$. To obtain $m$ terms in the asymptotic expansion, we multiply
the polynomials $P_m(-1/z)$ and $P_m(1/z)$ occurring
in~\eqref{eq:k0series}--\eqref{eq:i0series}, then discard half the terms
(here $z = 1/x$ is small when $x \approx 2n$ is large, so we discard the
terms involving high powers of $z$). To bound the error,
we show in Lemma~\ref{lem:T_4n_bd} 
that the discarded terms are sufficiently small, and also take into
account the error terms $R_m$ and $Q_m$ in the asymptotic expansions for
$K_0$ and $I_0$.

The main result, Theorem~\ref{thm:Jbound}, is given in
Section~\ref{sec:complete}.
Provided the parameter $N$ (the number of terms used to approximate
$S_0(2n)$ and $I_0(2n)$) is sufficiently large, the error is bounded
by $24e^{-8n}$.  Corollary~\ref{cor:simpler_bound} shows that it is sufficient
to take $N \approx 4.971n$.

\section{Bounds for the individual Bessel functions}	\label{sec:bounds}

Asymptotic expansions for $I_0(x)$ and $K_0(x)$
are given by Olver~\cite[pp.~266--269]{Olver1997} and
can be found in~\cite[\S10.40]{DLMF}.
They can be written as
\begin{equation}
K_0(x) =  e^{-x} \left(\frac{\pi}{2 x}\right)^{1/2} \left( P_m(-x) + R_m(x) \right)
\label{eq:k0series}
\end{equation}
and
\begin{equation}
I_0(x) = \frac{e^x}{(2 \pi x)^{1/2}} \left( P_m(x) + Q_m(x) \right),
\label{eq:i0series}
\end{equation}
where $R_m(x)$ and $Q_m(x)$ denote error terms,
\begin{equation}
P_m(x) = \sum_{k=0}^{m-1} a_k x^{-k}, 
	\;\;\text{and}\;\;
	a_k = \frac{[(2k)!]^2}{(k!)^3 32^k}\,\raisedot
\label{eq:pseries}
\end{equation}

For $n \ge 1$,
\begin{equation}				\label{eq:factorial_bounds}
\sqrt{2\pi} n^{n+1/2} e^{-n} \le n! \le e n^{n+1/2} e^{-n},
\end{equation}
so the coefficients $a_k$ in \eqref{eq:pseries} satisfy
\begin{equation}
a_k \le \frac{e^2}{\pi^{3/2} 2^{1/2}} \frac{1}{k^{1/2}} \left(\frac{k}{2e}\right)^k
< \frac{1}{k^{1/2}} \left(\frac{k}{2e}\right)^k
\label{eq:coeffbound}
\end{equation}
for $k \ge 1$ (the first term is $a_0 = 1$).

For $x > 0$, we also have the global bounds
\begin{equation}
0 < K_0(x) < e^{-x} \left(\frac{\pi}{2 x}\right)^{1/2}
\label{eq:globalbound1}
\end{equation}
and
\begin{equation}
I_0(x) > \frac{e^x}{(2 \pi x)^{1/2}}\,\raisedot
\label{eq:globalbound2}
\end{equation}

Observe that the bound on $K_0(x)$ and equation~\eqref{eq:k0series} imply that
\begin{equation}			\label{eq:PR_ineq}
|P_m(-x) + R_m(x)| < 1.
\end{equation}

For $x > 0$, the series~\eqref{eq:k0series} for $K_0(x)$
is alternating, and the remainder satisfies
\begin{equation}				\label{eq:R_mbound}
|R_m(x)| \le \frac{a_m}{x^m}
	< \frac{1}{m^{1/2}} \left(\frac{m}{2e}\right)^m
		\frac{1}{x^m}\,\raisedot
\end{equation}

The series~\eqref{eq:i0series} for $I_0(x)$ is not alternating.
The following lemma bounds the error $Q_m(x)$.
\begin{lemma}		\label{lemma:Q_mbound}
Let $Q_m(x)$ be defined by~\eqref{eq:i0series}. Then for $m \ge 1$ and
real $x \ge 2$ we have
\[|Q_m(x)| \le 4\left(\frac{m}{2ex}\right)^m + e^{-2x}.\]
\end{lemma}
\begin{proof}
The identity	
$I_0(x) = i(K_0(-x) - K_0(x))/\pi$
gives
\begin{equation}
Q_m(x) = R_m(-x) - \frac{i}{\pi} \frac{(2\pi x)^{1/2}}{e^x} K_0(x).
\end{equation}

According to Olver~\cite[p.~269]{Olver1997},
\begin{equation}
|R_m(-x)| \le 2 \chi(m) \exp(\tfrac{1}{8} \pi x^{-1}) a_m x^{-m},
\end{equation}
where
\begin{equation}
\chi(m) = \pi^{1/2} \frac{\Gamma(m/2+1)}{\Gamma(m/2+1/2)}
	\le \frac{\pi}{2}\,m^{1/2}
\end{equation}
(the bound on $\chi(m)$ follows as $\chi(m)/m^{1/2}$ is monotonic decreasing
for $m \ge 1$).

Since $x \ge 2$, applying \eqref{eq:coeffbound} gives
\begin{equation}
|R_m(-x)| \le \pi e^{\pi/16} \left(\frac{m}{2e}\right)^m \frac{1}{x^m}
	< 4 \left(\frac{m}{2ex}\right)^m.
\end{equation}

Combined with the global bound \eqref{eq:globalbound1} for $K_0(x)$, we obtain
\begin{equation}		\label{eq:Q_m_bd}
|Q_m(x)| \le |R_m(-x)| + \frac{1}{\pi} \frac{(2\pi x)^{1/2}}{e^x} K_0(x)
 \le 4 \left(\frac{m}{2ex}\right)^m + e^{-2x}.
\end{equation}
\end{proof}

\pagebreak[3]

\begin{corollary}			\label{cor:IK_upper_bd}
For $x \ge 2$, we have $0 < I_0(x)K_0(x) < 1/x$.
\end{corollary}
\begin{proof}
The first inequality is obvious, since both $I_0(x)$ and $K_0(x)$
are positive.
Also, using~\eqref{eq:i0series} and~\eqref{eq:Q_m_bd} with $m=1$ gives
\[I_0(x) \le \frac{e^x}{(2\pi x)^{1/2}}(1 + e^{-1} + e^{-4}),\]
so from \eqref{eq:globalbound1} we have
\[I_0(x)K_0(x) \le \frac{1 + e^{-1} + e^{-4}}{2x} < \frac{1}{x}\,\raisedot\]
\end{proof}

\begin{lemma}					\label{lemma:R_and_Sbound}
If $R_m(x)$ and $Q_m(x)$ are defined by~\eqref{eq:k0series}
and~\eqref{eq:i0series} respectively, then
\begin{equation}
|R_{4n}(2n)| \le \frac{e^{-4n}}{2 n^{1/2}}
\;\;\text{and}\;\; |Q_{4n}(2n)| \le 5 e^{-4n}.
\label{eq:rsevalbounds}
\end{equation}
\end{lemma}
\begin{proof}
Taking $x = 2n$ and $m = 4n$, the inequality~\eqref{eq:R_mbound}
gives the first inequality, and 
Lemma~\ref{lemma:Q_mbound} gives the second inequality.
\end{proof}

We also need the following lemma.
\begin{lemma}
If $P_m(x)$ is defined by~\eqref{eq:pseries}, then
\begin{equation}
|P_{4n}(2n)| < 2\;\;\text{and}\;\; |P_{4n}(-2n)| < 1.
\label{eq:pnbounds}
\end{equation}
\end{lemma}
\begin{proof}
Using~\eqref{eq:pseries} and~\eqref{eq:coeffbound}, we have
\begin{eqnarray*}
P_{4n}(2n) &=& 1 + \sum_{k=1}^{4n-1}\frac{a_k}{(2n)^k}\\
	   &\le& 1 + \sum_{k=1}^{4n-1}k^{-1/2}\left(\frac{k}{4en}\right)^k\\
	   &\le& 1 + \sum_{k=1}^{4n-1}e^{-k} < \frac{e}{e-1} < 2.
\end{eqnarray*}
The right inequality in \eqref{eq:pnbounds}
can be proved in a similar manner, taking the sign alternations
into account.
\end{proof}

\section{Bounds for the product}		\label{sec:product}

We wish to bound the error term $T_m(x)$ in \eqref{eq:ikasymp}
when evaluated at $x = 2n$, $m = 4n$. The result is given by the 
following lemma.
\begin{lemma}			\label{lem:T_4n_bd}
If $T_m(x)$ is defined by~$\eqref{eq:ikasymp}$, then
$T_{4n}(2n) < 7e^{-4n}$.
\end{lemma}
\begin{proof}
In terms of the
expansions for $I_0(x)$ and $K_0(x)$, we have
\begin{eqnarray}
2x I_0(x) K_0(x) &=& (P_m(-x) + R_m(x)) (P_m(x) + Q_m(x))\nonumber \\
	&=& P_m(x) P_m(-x) +\nonumber \\
	&& \left[(P_m(-x)+R_m(x)) Q_m(x) + P_m(x) R_m(x)\right].
\label{eq:expanded}
\end{eqnarray}

It follows from \eqref{eq:PR_ineq},
\eqref{eq:rsevalbounds} and \eqref{eq:pnbounds} that
the expression $[\cdots]$ in \eqref{eq:expanded},
evaluated at $x = 2n$, $m = 4n$,
is bounded in absolute value by
\begin{equation}
5 e^{-4n} + e^{-4n}/n^{1/2} \le 6e^{-4n}.
\label{eq:crossbound}
\end{equation}

Next, we rewrite
\begin{equation*}
P_m(x) P_m(-x) = \sum_{i=0}^{m-1} \sum_{j=0}^{m-1} (-1)^i a_i a_j x^{-(i+j)}
\end{equation*}
as $L + U$, where
\begin{equation}
L = \sum_{k=0}^{m-1} \left( \sum_{j=0}^k (-1)^j a_j a_{k-j} \right) x^{-k}
\label{eq:lsum}
\end{equation}
and
\begin{equation}
U = \sum_{k=m}^{2m-2} \left( \sum_{j=k-(m-1)}^{m-1} (-1)^j a_j a_{k-j} \right) 
	x^{-k}.
\label{eq:usum}
\end{equation}
The ``lower'' sum $L$ is precisely $\sum_{k=0}^{m/2-1} b_k x^{-2k}$.
Replacing $k$ by $2k$ in \eqref{eq:lsum} (as the odd terms
vanish by symmetry), we have to prove
\begin{equation}
\sum_{j=0}^{2k} \frac{(-1)^j [(2j)!]^2 [(4k-2j)!]^2}{(j!)^3 [(2k-j)!]^3 32^{2k}} = \frac{[(2k)!]^3}{(k!)^4 8^{2k}}
	\,\raisedot
\label{eq:hypsum}
\end{equation}
This can be done algorithmically using the creative telescoping
approach of Wilf and Zeilberger. For example, the
implementation in the Mathematica package \texttt{HolonomicFunctions}
by Koutschan \cite{Koutschan2010} can be used.
The command
\begin{verbatim}
    a = ((2j)!)^2 / ((j!)^3 32^j);
    CreativeTelescoping[(-1)^j a (a /. j -> 2k-j),
        {S[j]-1}, S[k]]       
\end{verbatim}
outputs the recurrence equation
\begin{equation*}
(8+8 k) b_{k+1} - \left(1+6 k+12 k^2+8 k^3\right) b_k = 0
\end{equation*}
matching the right-hand side of~\eqref{eq:hypsum},
together with a telescoping certificate.
Since the summand in \eqref{eq:hypsum} vanishes
for $j < 0$ and $j > 2k$, no boundary conditions
enter into the telescoping relation,
and checking the initial value ($k = 0$)
suffices to prove the identity.\footnote%
{Curiously, the built-in
\texttt{Sum} function in Mathematica 9.0.1
computes a closed form for the sum \eqref{eq:hypsum},
but returns an answer that is wrong by a factor 2
if the factor $[(4k-2j)!]^2$ in the summand is input as $[(2(2k-j))!]^2$.}

It remains to bound the ``upper'' sum $U$ given by \eqref{eq:usum}.
The coefficients of $U = \sum_{k=m}^{2m-2} c_k x^{-k}$
can also be written as
\begin{equation}
c_k = \sum_{j=1}^{2m-k-1} (-1)^{j+k+m} a_{k-m+j} a_{m-j}.
\end{equation}
By symmetry, this sum is zero when $k$ is odd, so we only need
to consider the case of $k$ even.
We first note that, if $1 \le i < j$, then $a_i a_j \ge a_{i+1} a_{j-1}$. This
can be seen by observing that the ratio satisfies
\begin{equation}
\frac{a_i a_j}{a_{i+1} a_{j-1}} = \frac{(i+1) (2j-1)^2}{j (2i+1)^2} \ge 1.
\end{equation}
Thus, after adding the duplicated terms, $c_k$ can be written as an alternating sum in which
the terms decrease in magnitude, e.g.
\begin{equation}
-2 a_1 a_{11} + 2 a_2 a_{10} - \ldots + 2 a_5 a_7 - a_6 a_6,
\end{equation}
and its absolute value can be bounded by that of the first term, $2 a_{1+k-m} a_{m-1}$, giving
\begin{equation}
\left|\sum_{k=m}^{2m-2} \frac{c_k}{x^k} \right| \le \sum_{k=m}^{2m-2} t_k,
	\quad t_k = \frac{2 a_{1+k-m} a_{m-1}}{x^k}\,\raisedot
\end{equation}

Evaluating at $x = 2n, m = 4n$ as usual, the term ratio
\begin{equation}
\frac{t_{k+1}}{t_k} = \frac{(3+2k-8n)^2}{16n(2+k-4n)}
\end{equation}
is bounded by 1 when $4n \le k \le 8n-2$. Therefore,
using~\eqref{eq:coeffbound},
\begin{equation}
\sum_{k=m}^{2m-2} t_k \le (m-1) t_m \le e^{-4n} \frac{(4n-1)^{4n-1/2}}{2^{8n-1} n^{4n}} < e^{-4n}.
\label{eq:tbound}
\end{equation}
Adding \eqref{eq:crossbound} and \eqref{eq:tbound}, we find that
$|T_{4n}(2n)| < 7e^{-4n}$.
\end{proof}

\section{A complete error bound}	\label{sec:complete}

We are now equipped to justify Algorithm~B3.
The algorithm computes an approximation $\widetilde{\gamma}$ to $\gamma$.
Theorem~\ref{thm:Jbound} bounds the error
$|\widetilde{\gamma}-\gamma|$ in the algorithm, excluding
rounding errors and any error in the evaluation of $\ln n$.
The finite sums $S$ and $I$ approximate $S_0(2n)$ and
$I_0(2n)$ respectively, while $T$ approximates $I_0(2n)K_0(2n)$.

\begin{theorem}					\label{thm:Jbound}
Given an integer $n \ge 1$, let $N \ge 4n$ be an
integer such that
\begin{equation}
\frac{2 n^{2N} H_N}{(N!)^2} < \varepsilon_0,
\label{eq:e0bound}
\end{equation}
where
\begin{equation}
\varepsilon_0 = \frac{e^{-6n}}{(4\pi n)^{1/2} (1+H_N)}\,\raisedot
\label{eq:eps0_def}
\end{equation}
Let
\begin{equation*}
S = \sum_{k=0}^{N-1} \frac{H_k n^{2k}}{(k!)^2}\,\raisecomma
\quad I = \sum_{k=0}^{N-1} \frac{n^{2k}}{(k!)^2}\,\raisecomma
\quad T = \frac{1}{4n} \sum_{k=0}^{2n-1} 
\frac{[(2k)!]^3}{(k!)^4 8^{2k} (2n)^{2k}}\,\raisecomma
\end{equation*}
and
\begin{equation*}
\widetilde{\gamma} = \frac{S}{I} - \frac{T}{I^2} - \ln n\,.
\end{equation*}
Then
\begin{equation}
|\widetilde{\gamma} - \gamma| < 24 e^{-8n}.		\label{eq:fullbound}
\end{equation}
\end{theorem}

\begin{proof}

Let
\begin{align*}
\varepsilon_1 = S_0(2n) - S & = \sum_{k=N}^{\infty} \frac{H_k n^{2k}}{(k!)^2}
	\,\raisecomma \\
\varepsilon_2 = I_0(2n) - I & = \sum_{k=N}^{\infty} \frac{n^{2k}}{(k!)^2}
	\,\raisedot
\end{align*}
Inspection of the term ratios for $k \ge N$ shows that
$\varepsilon_1$ and $\varepsilon_2$ are bounded by
the left side of~\eqref{eq:e0bound}.
Using~\eqref{eq:globalbound2} to bound $1/I_0(2n)$, it follows that
\begin{align*}
\left| \frac{S+\varepsilon_1}{I+\varepsilon_2} - \frac{S}{I} \right|
& = \left| \frac{\varepsilon_1 I - \varepsilon_2 S}{(I + \varepsilon_2) I} \right| \\
& \le \frac{\varepsilon_0 (I + S)}{(I + \varepsilon_2) I} \\
& = \varepsilon_0 \left( \frac{1}{I_0(2n)} \right) \left( 1 + \frac{S}{I} \right) \\
& < \frac{e^{-6n}}{(4\pi n)^{1/2} (1+H_N)} \left( \frac{(4 \pi n)^{1/2}}{e^{2n}} \right) ( 1 + H_N ) \\
& = e^{-8n}.
\end{align*}
We have $T + \varepsilon_3 = I_0(2n) K_0(2n)$ where,
from Lemma~\ref{lem:T_4n_bd},
$|\varepsilon_3| < 7e^{-4n} / (4n)$.
Thus, from Corollary~\ref{cor:IK_upper_bd}, 
\[T \le \frac{1}{2n} + \frac{7e^{-4n}}{4n} < \frac{1}{n}\,\raisedot\]
Therefore, using~\eqref{eq:globalbound2} again,
\begin{align*}
\left| \frac{T+\varepsilon_3}{(I+\varepsilon_2)^2} - \frac{T}{I^2} \right|
& = \left| \frac{\varepsilon_3 I^2 - T \varepsilon_2 (2 I + \varepsilon_2)}{(I + \varepsilon_2)^2 I^2} \right| \\
& \le \frac{|\varepsilon_3|}{(I+\varepsilon_2)^2} + T \varepsilon_2 \frac{(2I+\varepsilon_2)}{(I+\varepsilon_2)^2 I^2} \\
& \le \frac{|\varepsilon_3|}{I_0(2n)^2} + T \varepsilon_2 \frac{3}{I_0(2n)^3} \\
& < 7 \pi e^{-8n} + e^{-8n} \\
& < 23 e^{-8n}.
\end{align*}
Thus, the total error $|\widetilde{\gamma} - \gamma|$
is bounded by $e^{-8n} + 23e^{-8n} = 24 e^{-8n}$.
\end{proof}

\begin{remark}			\label{remark:1}
{\rm
We did not try to obtain the best possible constant
in~\eqref{eq:fullbound}. A more detailed analysis shows that we can
reduce the constant $24$ 
by a factor greater
than two if $n$ is large. See also Remark~\ref{remark:3}.
}
\end{remark}

Since the condition on $N$ in Theorem~\ref{thm:Jbound} is rather
complicated, we give the following corollary.
\begin{corollary}			\label{cor:simpler_bound}
Let $\alpha \approx 4.970625759544$ be the
unique positive real solution of $\alpha (\ln \alpha - 1) = 3$.
If $n \ge 138$ and $N \ge \alpha n$ are integers, then the conclusion
of Theorem~$\ref{thm:Jbound}$ holds.
\end{corollary}
\begin{proof}

For $138 \le n \le 214$
we can verify by direct computation that 
conditions~\eqref{eq:e0bound}--\eqref{eq:eps0_def}
of Theorem~\ref{thm:Jbound} hold.
Hence, in the following we assume that $n \ge 215$.
Since $N \ge \alpha n$, this implies that 
$N \ge \lceil 215\alpha \rceil = 1069$.

Let $\beta = N/n$. Then $\beta \ge \alpha$, so 
$\beta(\ln\beta - 1) \ge 3$.
Thus $2n(\beta\ln\beta - \beta - 3) \ge 0$. Taking exponentials
and using $\beta = N/n$, we obtain
\begin{equation}		\label{key_ineq}
N^{2N} \ge e^{2N+6n}n^{2N}.
\end{equation}

Define the real analytic function $h(x) := \ln x + \gamma + 1/(2x)$.
The upper bound $H_N \le h(N)$
follows from the Euler-Maclaurin expansion
\[H_N - \ln(N) - \gamma \sim \frac{1}{2N} - 
  \sum_{k=1}^\infty \frac{B_{2k}}{2k} N^{-2k},
\]
since the terms on the right-hand-side alternate in sign.

Using our assumption that $N \ge 1069$, it is easy to verify that
\begin{equation}			\label{eq:N_ineq2}
\sqrt{\pi\alpha N} \ge 2h(N)(h(N)+1).
\end{equation}
Since $\beta \ge \alpha$, it follows from~\eqref{eq:N_ineq2} that
\begin{equation}			\label{eq:N_ineq1}
\sqrt{\pi\beta N} \ge 2h(N)(h(N)+1).
\end{equation}
Substituting $\beta = N/n$ in~\eqref{eq:N_ineq1},
it follows that
\begin{equation}			\label{eq:Nn_ineq}
\pi N > 2 h(N)(h(N)+1)(\pi n)^{1/2}.
\end{equation}
Using~\eqref{key_ineq}, this gives
\begin{equation}			\label{eq:36b}
\pi N^{2N+1} > 2n^{2N}h(N)(h(N)+1)
	(\pi n)^{1/2}e^{2N+6n}.
\end{equation}

{From} the first inequality of~\eqref{eq:factorial_bounds} we have
$(N!)^2 \ge {2\pi}N^{2N+1}e^{-2N}$. 
Using this and $h(N) \ge H_N$, we see that~\eqref{eq:36b} implies
\begin{equation}				\label{eq:e0bound2}
(N!)^2 > 4n^{2N}H_N(1+H_N)(\pi n)^{1/2}e^{6n}.
\end{equation}
However, it is easy to see that~\eqref{eq:e0bound2} is equivalent
to conditions~\eqref{eq:e0bound}--\eqref{eq:eps0_def}
of Theorem~\ref{thm:Jbound}. Hence, the conclusion of
Theorem~\ref{thm:Jbound} holds.
\end{proof}

\begin{remark}			\label{remark:2}
{\rm
If $0 < n < 138$ then Corollary~$\ref{cor:simpler_bound}$ does not apply,
but a numerical computation shows that it is always sufficient to take
$N \ge \alpha n + 1$.
}
\end{remark}

\begin{remark}			\label{remark:3}
{\rm
As indicated in Table~\ref{tab:bcomparison}, the bound in \eqref{eq:fullbound}
is nearly optimal for large~$n$.
Our bound $24e^{-8n}$ appears to overestimate the true error
by a factor that grows slightly faster than order $n^{1/2}$,
which is inconsequential for high-precision computation of~$\gamma$.
}
\end{remark}

\begin{table}[ht]
\begin{center}
\begin{tabular}{ c | c | l | l }
$n$ & $N$ & \;\;\;$|\widetilde{\gamma} - \gamma|$ & \;\;\;$24e^{-8n}$\\[2pt]
\hline
&&&\\[-9pt]
10 & 50 & $7.68 \cdot 10^{-38}$ & $4.34 \cdot 10^{-34}$ \\
100 & 498 & $5.32 \cdot 10^{-349}$ & $8.81 \cdot 10^{-347}$ \\
1000 & 4971 & $1.96 \cdot 10^{-3476}$ & $1.06 \cdot 10^{-3473}$ \\
10000 & 49706 & $2.85 \cdot 10^{-34746}$ & $6.64 \cdot 10^{-34743}$
\end{tabular}
\caption{The error $|\widetilde{\gamma} -\gamma|$ 
compared to the bound \eqref{eq:fullbound}.}
\label{tab:bcomparison}
\end{center}
\end{table}

\bibliographystyle{plain}

\pagebreak[3]

\end{document}